\documentclass[a4paper,12pt]{amsart} 

\usepackage[right=2cm, left=2cm, top=2cm, bottom=2cm]{geometry}

\usepackage{amssymb,amsmath}
\usepackage{tikz-cd}

\newtheorem{thm}{Theorem}[section] 
\newtheorem{prop}[thm]{Proposition}
\newtheorem{cor}[thm]{Corollary} 
\newtheorem{lem}[thm]{Lemma}

\theoremstyle{definition}

\numberwithin{equation}{section}

\newcommand{\skal}[2]{\langle #1,#2\rangle}
\newcommand{\alg}[1]{\mathfrak{#1}}

\begin{document}

\title{Lifts of maps to frame bundles}
\author{Kamil Niedzia\l omski and Malgorzata Niedzia\l omska}

\date{}

\address{
Department of Mathematics and Computer Science \endgraf
University of \L\'{o}d\'{z} \endgraf
ul. Banacha 22, 90-238 \L\'{o}d\'{z} \endgraf
Poland
}
\email{kamil.niedzialomski@wmii.uni.lodz.pl}
\email{malgorzata.niedzialomska@wmii.uni.lodz.pl}

\begin{abstract}
Let $(M,g)$ be a Riemannian manifold, $L(M)$ be its frame bundle, $O(M)$ its orthonormal frame bundle. For a distribution $D$ on $M$ we define a subbundle $L(D)\subset L(M)$ or $O(D)\subset O(M)$ in a natural way. This allows us to consider a lift $L\varphi$ of a map $\varphi:M\to N$ not necessarily being a local diffeomorphism. More precisely, if $\varphi:M\to N$ is a submersion, then $L\varphi:L(\mathcal{H}^{\varphi})\to L(N)$ or $L\varphi:O(\mathcal{H}^{\varphi})\to L(N)$, where $\mathcal{H}^{\varphi}$ is a horizontal distribution of $\varphi$. Equipping $L(M)$ and $L(N)$ with the Mok metrics, we study conformality and harmonicity of lifts $L\varphi$. 
\end{abstract}

\maketitle

\section{Introduction} 
Let $(M,g)$ be a Riemannian manifold, $L(M)$ its frame bundle. The first example of a Riemannian metric on $L(M)$ was considered by Mok \cite{mok}. This metric, called the Sasaki--Mok metric or the diagonal lift $g^d$ of $g$, was also investigated in \cite{cl1,cl2}. Later, Kowalski and Sekizawa \cite{ks1, ks2} introduced a family of natural metrics on the frame bundle. Thus, the geometry of a frame bundle and its subbundles can be considered. This has been done, for example, by the first author in \cite{n2, n3}. Moreover, we can study the geometric properties of maps between frame bundles. 

The most natural example of a map between frame bundles is a lift of a local diffeomorphism. Precisely, for a local diffeomorphism $\varphi:M\to N$, its lift $L\varphi:L(M)\to L(N)$ is given by the following formula
\begin{equation*}
L\varphi(u)=(\varphi_{\ast}u_1,\ldots,\varphi_{\ast}u_n),\quad u=(u_1,\ldots,u_n)\in L(M).
\end{equation*}
If $\varphi:M\to N$ is an immersion, then $\varphi$ is a local diffeomorphism onto its image $\varphi(M)$, which is a submanifold of $N$. Hence, in this case $L\varphi:L(M)\to L(\varphi(M))$. 

It is also possible to define the lift of a submersion. In deed, a submersion $\varphi:M\to N$ induces an isomorphism of the horizontal distribution $\mathcal{H}^{\varphi}=({\rm ker}\varphi_{\ast})^{\bot}$ and the tangent space $TN$ of $N$, $\varphi_{\ast}:\mathcal{H}^{\varphi}\to TN$. Moreover, for any distribution $D$ on $M$ of dimension $k$ we put
\begin{equation*}
L(D) = \{u = (u_1, \ldots, u_n)\in L(M)\mid u_1,\ldots, u_k\in D; u_{k+1},\ldots, u_n\in D^{\bot}\},
\end{equation*} 
Then $L(D)$ is a subbundle of $L(M)$ and the lift $L\varphi$ of $\varphi$ is a well-defined map on $L(\mathcal{H}^{\varphi})$. However, since we work in the Riemannian setting, it is maybe more convenient to consider the following subbundle of the orthonormal frame bundle $O(M)$:
\begin{equation*}
O(D) = \{u = (u_1, \ldots, u_n)\in O(M)\mid u_1,\ldots, u_k\in D; u_{k+1},\ldots, u_n\in D^{\bot}\}.
\end{equation*}
Clearly, $L\varphi$ restricts to $O(\mathcal{H}^{\varphi})$. 

In this paper, we study the geometry of the lifts $L\varphi$. First of all, we investigate conformality. We may consider two sets
\begin{align*}
    {\rm Conf}(L(M),L(N)) & = \{L\varphi:L(\mathcal{H}^{\varphi})\to L(N)\mid \textrm{$L\varphi$ is conformal}\}, \\
    {\rm Conf}(O(M),L(N)) & = \{L\varphi:O(\mathcal{H}^{\varphi})\to L(N)\mid \textrm{$L\varphi$ is conformal}\}
\end{align*} 
Clearly, we have an inclusion ${\rm Conf}(L(M),L(N)) \subset {\rm Conf}(O(M),L(N))$, but not necessary the equality. We will study the bigger (at least, not smaller) set. Thus we restrict our considerations to $O(M)$. Secondly, we consider the condition for a lift $L\varphi$ to be a harmonic morphism.

In order to achieve the conditions for conformality and harmonicity, we study properties of the subbundle $O(\mathcal{H}^{\varphi})$. We rely on the general results obtained by the second author for the geometry of $O(D)$ inside $O(M)$, see \cite{n2, n3}. However, we prove some additional results. Using the description of the horizontal and vertical distributions of $O(\mathcal{H}^{\varphi})$ we show how horizontal and vertical lifts behave under the action of the differential of a lift $L\varphi$. Then, we describe the vertical and horizontal distributions: $\mathcal{V}^{L\varphi}$ and $\mathcal{H}^{L\varphi}$. To determine the latter, we need a Riemannian metric on $L(M)$. We consider the most natural one, a type of well known Mok metric.

The main part of the paper contains two theorems - characterization of conformality andcondition for the lift to be a harmonic morphism.

\subsection{Motivation}

The motivation for all the considerations in this article comes from the analogous considerations for the lifts of maps to the tangent bundle considered by the second author and W. Kozlowski in \cite{kn1, kn2}. In this case, the differential $\varphi_{\ast}:TM\to TN$ of a submersion $\varphi:M\to N$ is again a well-defined submersion. Equipping $TM$ with natural Riemannian metrics conformality and harmonicity of $\varphi_{\ast}$ can be considered. Results in this note are a natural generalization of mentioned results for the tangent bundle.

\subsection{Notation}

Adopt the following index convention. The dimension of a base manifold $M$ is $n$, whereas the dimension of a distribution $D$ and $\mathcal{H}^{\varphi}$ on $M$ is $k$. Indices $A,B,C$ run from $1$ to $k$, indices $\alpha,\beta,\gamma$ from $k+1$ to $n$ and indices $i,j,l$ from $1$ to $n$.

\section{Preliminary results}

\subsection{Frame bundles}

Let $M$ be an $n$--dimensional manifold and let $\pi_{L(M)}:L(M)\to M$ be its frame bundle. Thus, $L(M)$ consists of all bases $(u_1,\ldots,u_n)$ of tangent spaces $T_xM$, $x\in M$. The Lie group $GL(n)$ acts naturally from the right on elements of $L(M)$.

Let $\nabla$ be a linear connection on $M$. It defines the horizontal distribution $\mathcal{H}^{L(M)}$ in $L(M)$. Together with the vertical distribution $\mathcal{V}^{L(M)}$, being the kernel of the differential of the projection $\pi_{L(M)}$, we have the splitting
\begin{equation*}
T_uL(M) = \mathcal{H}^{L(M)}_u \oplus \mathcal{V}^{L(M)}_u,\quad u\in L(M).
\end{equation*}
Moreover, any vector $X\in T_xM$ has the unique lift $X^h_u$ to $\mathcal{H}^{L(M)}_u$, where $\pi_{L(M)}(u)=x$. The lift $X^h$ is called the {\it horizontal lift} of $X$. There are also natural vector fields associated with the vertical distribution. Let $A\in\alg{gl}(n)$ and denote by $A^{\ast}$ the following vector field
\begin{equation*}
A^{\ast}_u = \frac{d}{dt}\left(u\cdot {\rm exp}(tA)\right)_{t=0}.
\end{equation*}
Since the curve $t\mapsto u\cdot {\rm exp}(tA)$ lies in the fiber $L(M)_x$, where $\pi_{L(M)}(u)=x$, it follows that $A^{\ast}$ is a vertical vector called {\it fundamental vertical vector}. There is an alternative descriptions of the vertical distribution \cite{n2, n3}, which we will use. Let $P$ be an endomorphism of the tangent bundle of $M$. Treating a frame $u\in L(M)$ as a linear isomorphism $u:\mathbb{R}^n\to T_xM$, $\pi(u)=x$, a composition $u^{-1}\circ P\circ u$ is an endomorphism of $\mathbb{R}^n$, i.e. a matrix in $\alg{gl}(n)$. Thus we may consider a fundamental vertical vector of such matrix at a frame $u$. Denote it by $P^{\ast}_u$. Thus
\begin{equation*}
    P^{\ast}_u = (u^{-1}\circ P\circ u)^{\ast}_u.
\end{equation*}

In order to study the geometry of submersions between frame bundles, we need to introduce Riemannian metrics on the frame bundles. There are many ways to do it \cite{mok,ks0,n1}. We will introduce and study the most natural Riemannian metric, in which horizontal and vertical distributions are orthogonal. Let $g_M$ be a Riemannian metric on $M$. Consider on $L(M)$ the following Riemannian metric $g_{L(M)}$:
\begin{align*}
    g_{L(M)}(X^h_u, Y^h_u) &= g_M(X,Y) \\
    g_{L(M)}(X^h_u, Q^{\ast}_u) &= 0 \\
    g_{L(M)}(P^{\ast}_u, Q^{\ast}_u) &= \sum_i g_M(P(u_i), Q(u_i)),
\end{align*}
where $P,Q$ are endomorphisms of $TM$. This metric is often called  Mok metric or diagonal metric \cite{mok}.

Let us state the formula for the Levi-Civita connection of $g_{L(M)}$. Let $\nabla$ be the Levi-Civita connection of $g_M$ on $M$ and let $R$ be the curvature tensor with the convention $R(X,Y) = [\nabla_X,\nabla_Y]-\nabla_{[X,Y]}$. Firstly, we have \cite{n1}
\begin{align*}
    [X^h, Y^h] &= [X,Y]^h - R(X,Y)^{\ast},\\
    [X^h, Q^{\ast}] &= -(\nabla_XP)^{\ast},\\
    [P^{\ast}, Q^{\ast}] &= -[P,Q]^{\ast}.
\end{align*}
Moreover, the Levi-Civita connection $\nabla^{L(M)}$ of $g_{L(M)}$ satisfies \cite{n2, n3}:
\begin{align*}
     \nabla^{L(M)}_{X^h} Y^h &=(\nabla_XY)^h - \frac{1}{2}R(X,Y)^{\ast}\\
     \nabla^{L(M)}_{X^h} Q^{\ast} &=\frac{1}{2}R_Q(X)^h,\\
     \nabla^{L(M)}_{P^{\ast}} Y^h &=\frac{1}{2}R_P(Y)^h + (\nabla_XP)^{\ast},\\
     \nabla^{L(M)}_{P^{\ast}} Q^{\ast} &= (Q\circ P)^{\ast},
\end{align*}
where $R_P$ is a curvature type endomorphism of the form
\begin{equation*}
    R_P(X) = \sum_i R(e_i, P(e_i))X,\quad X\in TM.
\end{equation*}

We will consider also orthonormal frame bundle $\pi_{O(M)}:O(M)\to M$. It is a subbundle of $L(M)$ with a structure group $O(n)$. Its Lie algebra of skew--symmetric matrices is denoted by $\alg{so}(n)$. The connection form $\omega$ of the Levi--Civita connection reduces to a connection form on $O(M)$. Thus, the horizontal distribution $\mathcal{H}^{O(M)}_u$ at frames $u\in O(M)$ coincides with the horizontal distribution $\mathcal{H}^{L(M)}_u$. The vertical distribution $\mathcal{V}^{O(M)}_u$ is spanned by elements $P^{\ast}_u$, where $P$ is a skew-symmetric endomorphism, i.e. $P\in \alg{so}(TM) = O(M)\times_{{\rm ad} O(n)} \alg{so}(n)$, or alternatively,
\begin{equation*}
    g(P(X),Y) = -g(X, P(Y)),\quad X,Y\in TM.
\end{equation*}

Let us describe the Levi-Civita connection of $g_{L(M)}$ on $O(M)$. The first three formulas for the Levi-Civita connection of $g_{L(M)}$ are also valid when restricting to $O(M)$. The only difference is for the value of a Levi-Civita connection $\nabla^{O(M)}$ on vertical vector fields:
\begin{equation*}
    \nabla^{O(M)}_{P^{\ast}} Q^{\ast} = -\frac{1}{2}[P,Q]^{\ast},\quad P,Q\in\alg{so}(TM).
\end{equation*}

\subsection{Frame bundle adapted to a distribution}

Let $(M,g)$ be a $n$--dimensional Riemannian manifold and $\pi_{L(M)}:L(M)\to M$ its frame bundle and $\pi_{O(M)}:O(M)\to M$ its orthonormal frame bundle. Let $D$ be a $k$--dimensional distribution on $M$. Denote its orthogonal complement by $D^{\bot}$. We define a subbundle of the orthonormal frame bundle associated with $D$ by
\begin{equation*}
O(D) = \{u = (u_1, \ldots, u_n)\in O(M)\mid u_1,\ldots, u_k\in D; u_{k+1},\ldots, u_n\in D^{\bot}\}
\end{equation*}
Consider the following group
\begin{equation*}
G = \left\{ \left( \begin{array}{cc} a & 0 \\ 0 & b \end{array} \right):\, a\in O(k), b\in O(n-k) \right\}.
\end{equation*}
Then  $G$ acts transitively on $O(D)$. Thus, $\pi_{O(D)}:O(D)\to M$ is a subbundle of $O(M)$ with a structure group $G$. Denoting by $\alg{g}$ the Lie algebra of $G$, we have $\alg{so}(n)=\alg{g}\oplus\alg{m}$, where
\begin{align*}
\alg{g} &=\left\{\left(\begin{array}{cc} A & 0 \\ 0 & B \end{array}\right):\, A\in\alg{so}(k), B\in \alg{so}(n-k) \right\},\\
\alg{m} &=\left\{\left(\begin{array}{cc} 0 & C \\ -C^{\top} & 0 \end{array}\right):\, C\in\mathcal{M}_{(n-k)\times k}\right\}.
\end{align*}
Notice that
\begin{equation*}
[\alg{g},\alg{n}]\subset\alg{m}.
\end{equation*}
i.e., the decomposition $\alg{so}(n)=\alg{g}\oplus\alg{n}$ is reductive. This is important in the context of inducing connection forms. Moreover
\begin{equation*}
TM = O(D)\times_{G}\mathbb{R}^n,
\end{equation*}
i.e. the tangent bundle is a vector bundle associated with $O(D)$. Notice that the adjoint bundle
\begin{equation*}
\alg{g}(TM) = O(M)\times_{{\rm ad}(G)}\alg{g}\subset {\rm End}(TM)
\end{equation*}
is the bundle of endomorphisms of the tangent bundle $TM$ which leave $D$ and $D^{\bot}$ invariant, Moreover, for an endomorphism $P:TM\to TM$, denote by $P^{\top}$ its part restricted to $D$, i.e., $P^{\top}:D\to D$ is given by the condition
\begin{equation*}
    P(X) = P^{\top}(X) + P'(X),\quad X\in D,
\end{equation*}
where the sum is written with respect to the splitting $TM = D \oplus D^{\bot}$. Analogously, we define $P^{\bot}:D^{\bot}\to D^{\bot}$,
\begin{equation*}
    P(X) = P''(X) + P^{\bot}(X),\quad X\in D{^\bot}.
\end{equation*}
In matrix block notation with respect to the splitting $TM = D\oplus D^{\bot}$
\begin{equation*}
    P = \left(\begin{array}{cc}
    P^{\top} & P' \\ P'' & P^{\bot}
    \end{array}\right).
\end{equation*}
Hence, $P^{\top}\in\alg{g}(TM)$ if and only if $P\in\alg{so}(D)$, $P^{\bot}\in\alg{so}(D^{\bot})$, $P'=0$ and $P''=0$. 

\subsection{Natural connection associated with a distribution on a Riemannian manifold}
Let us define a linear connection associated with a distribution $D$. Firstly, let $\nabla$ be the Levi--Civita connection associated with the fixed Riemannian metric $g_M$. Then any vector $X\in TM$ can be decomposed as
\begin{equation*}
    X=X^{\top}+X^{\bot}
\end{equation*} 
with respect to the $g_M$--orthogonal decomposition $TM=D\oplus D^{\bot}$. We follow \cite{n2}. Put
\begin{equation*}
\nabla^D_XY = (\nabla_X Y^{\top})^{\top} + (\nabla_X Y^{\bot})^{\bot}. 
\end{equation*}
It can be easily checked that it is a linear connection. Denote by $S$ the difference tensor of $\nabla$ and $\nabla^D$. Then
\begin{equation*}
S_XY = \nabla_XY - \nabla^D_XY = (\nabla_X Y^{\top})^{\bot} + (\nabla_X Y^{\bot})^{\top}.
\end{equation*}
The torsion tensor $T^D(X,Y)=\nabla^D_XY-\nabla^D_YX-[X,Y]$ of $\nabla^D$ equals
\begin{equation*}
T^D(X,Y) = -S_XY + S_YX = (\nabla_YX^{\top}-\nabla_XY^{\top})^{\bot} + (\nabla_YX^{\bot} - \nabla_XY^{\bot})^{\top}.
\end{equation*}
Notice that $T^D$ restricted to tangent vectors to $D$ is an integrability tensor of $D$, whereas $T^D$ restricted to vectors orthogonal to $D$ is an integrability tensor of $D^{\bot}$.

Denote by $R$ and $R^D$ the curvature tensors of $\nabla$ and $\nabla^D$, respectively, with the convention $R(X,Y)=[\nabla_X,\nabla_Y]-\nabla_{[X,Y]}$. By the fact that $\nabla_X = \nabla^D_X + S_X$ we have
\begin{align*}
R(X,Y)Z &=R^D(X,Y)Z + (\nabla^D_XS)_YZ - (\nabla^D_YS)_XZ + S_{T^D(X,Y)}Z _+ [S_X,S_Y]Z,
\end{align*}
where we put
\begin{equation*}
    (\nabla^D_XS)_YZ = \nabla^D_X(S_YZ) - S_{\nabla^D_XY}Z - S_X(\nabla^D_YZ).
\end{equation*}

Le us move to the description of connections $\nabla$ and $\nabla^D$ on the level of frame bundles. 

Let $\mathcal{H}^{O(M)}$ be the horizontal distribution of the connection induced by $\nabla$ and let $\mathcal{H}^{O(D)}$ be the horizontal distribution induced by $\nabla^D$. Let $X^h_u$ denote the horizontal lift to $\mathcal{H}^{O(M)}_u$ of a vector $X\in T_xM$, $x=\pi_{O(M)}(u)$ and, analogously consider the horizontal lift $X^{h,D}_u$ for $u\in O(D)$. 

The following lemma states the relation between $X^h$ and $X^{h,D}$.

\begin{lem}\label{lem:XhD}\cite{n1}
Let $X\in T_xM$ and $u\in O(D)$ be such that $\pi_{O(D)}(u)=x$. Then
\begin{equation}\label{eq:differencelifts}
X^h_u=X^{h,D}_u-(S_X)^{\ast}_u.
\end{equation} 
\end{lem}

It is important to note that $X^{h,D}_u$, for $u\in O(D)$, is tangent to $O(D)$. In fact, $\nabla^D$ induces the connection form in $O(D)$ associated with $\alg{g}$--component of the connection form $\omega$ of the Levi-Civita connection $\nabla$.

Let us state the above fact as a corollary to Lemma \ref{lem:XhD}.

\begin{cor}\label{cor:XhD_tangent}
The horizontal lift $X^{h,D}_u$, for $u\in O(D)$, is tangent to $O(D)$.
\end{cor}

The Levi-Civita connection $\nabla^{O(\mathcal{H}^{\varphi})}$ satisfies \cite{n1, n2}:
\begin{align*}
     \nabla^{O(D)}_{X^{h,D}} Y^{h,D} &=(\nabla_XY)^{h,D} - \frac{1}{2}R^D(X,Y)^{\ast}\\
     \nabla^{O(D)}_{X^{h,D}} Q^{\ast} &=\frac{1}{2}L_Q(X)^{h,D} + (\nabla_XP)^{\ast},\\
     \nabla^{O(D)}_{P^{\ast}} Y^h &=\frac{1}{2}L_P(Y)^{h,D},\\
     \nabla^{O(D)}_{P^{\ast}} Q^{\ast} &= -\frac{1}{2}[P,Q]^{\ast},
\end{align*}
where $L_P$ is an endomorphism of the form
\begin{equation*}
    L_P(X) = \mathcal{W}^{-1}(R_P(X) - \sum_i \skal{(\nabla_XP)_{\alg{m}}}{S_{e_i}}e_i),\quad X\in TM,
\end{equation*}
and $R^D$ is the curvature tensor of $\nabla^D$. $\mathcal{W}$ is an invertible endomorphism of the tangent bundle satisfying (see the next sections for more details):
\begin{equation*}
    g_M(X,\mathcal{W}(Y)) = g_{O(M)}(X^{h,D}, Y^{h,D}).
\end{equation*}

\subsection{The tangent bundle}
We will need some results concerning the tangent bundle.

The tangent bundle $\pi_{TM}:TM\to M$ of a manifold $M$ is an associated bundle with the frame bundle $L(M)$ with the fiber $\mathbb{R}^n$, $TM = L(M)\times_{GL(n0)}\mathbb{R}^n$, that is, any tangent vector $X\in T_xM$ is of the form $[u,\xi]$, where $u\in L(M)$ over $M$ and $\xi\in\mathbb{R}^n$: $X = \sum_i \xi_iu_i$. Let $\omega$ be the connection form on $L(M)$ of the Levi-Civita connection. Let $\gamma$ be a curve on $M$ and fix $u\in L(M)$. Denote by $\gamma^h_u$ the unique horizontal lift of $\gamma$ to $u\in L(M)$. Then, by definition, $[\gamma^h,\xi]$ is a horizontal curve on $TM$ for any $\xi\in\mathbb{R}^n$. In particular, if $\gamma^h_u(t)=(u_1(t),\ldots,u_n(t))$, then for $\xi=\xi_i$, we see that $t\mapsto u_i(t)$ is a horizontal curve in $TM$. The space of horizontal vectors in $TM$, i.e., tangent to horizontal curves, is denoted by $\mathcal{H}^{TM}$. Any vector $X\in T_xM$ has therefore the unique lift $X^{h,TM}_Z$ to $\mathcal{H}^{TM}_Z$, $Z\in T_xM$. Moreover, we have the vertical distribution $\mathcal{V}^{TM}={\rm ker}\pi_{TM\ast}$. Again, any vector $X\in T_xM$ has the unique lift $X^v_Z$ to the vertical distribution $\mathcal{V}^{TM}_Z$ given by $X^{v,TM}_Z = \frac{d}{dt}(t\mapsto Z + tX)_{t=0}$. Summing up, the tangent bundle to $TM$ splits as $TTM = \mathcal{H}^{TM}\oplus\mathcal{V}^{TM}$.

The horizontal and vertical distributions may be described also with the use of, so called, connection map $K:TTM\to M$ defined by Dombrowski \cite{dom}. It satisfies the condition
\begin{equation*}
K(Z_{\ast}X) = \nabla_XZ,
\end{equation*}
where we treat $Z$ as a map $Z:M\to TM$. Equivalently, if $\gamma:I\to M$ is a curve on $M$ such that $\dot{\gamma}(0) = X$ and $Z:I\to TM$ is any vector field along $\gamma$, then
\begin{equation*}
K(\dot{Z}(0)) = (\nabla_XZ)(0).
\end{equation*}
We have
\begin{align*}
& \pi_{TM\ast} X^{h,TM} = X, && K(X^{h,TM}) = 0,\\
& \pi_{TM\ast} X^{v,TM} = 0,\ && K(X^{v,TM}) = X.
\end{align*}

All above facts can be rewritten replacing $L(M)$ by $O(M)$ assuming $M$ is equipped with a Riemannian metric $g_M$. Consider the mappings $\pi^i:L(M)\to TM$ (resp. $\pi^i:O(M)\to TM$), $\pi^i(u)=u_i$, which associate to a frame its $i$--th vector. We may write
\begin{equation*}
\pi^i(u) = [u,\xi_i],\quad u\in L(M),
\end{equation*}
where $\xi_i\in\mathbb{R}^n$ is the $i$--th vector of the canonical basis. 

\begin{lem}\label{lem:pii}
Let $X\in TM$ and $P\in{\rm End}(TM)$ and $u\in O(M)$. Then
\begin{equation*}
    \pi^i_{\ast}X^h_u = X^{h,TM}_{u_i}\quad\textrm{and}\quad
    \pi^i_{\ast}P^{\ast}_u = P(u_i)^{v,TM}_{u_i}.
\end{equation*}
\end{lem}
\begin{proof}
Let $X$ be tangent to a curve $\gamma$ at $t=0$. With above notation
\begin{equation*}
\pi^i_{\ast}X^h_u = \pi^i_{\ast}\dot{\gamma^h_u}(0)=\frac{d}{dt}[\gamma^h_u(t),e_i]_{t=0} = X^{h,TM}_{u_i}.   
\end{equation*}
Moreover
\begin{align*}
    \pi^i_{\ast}P^{\ast}_u &= \pi^i_{\ast}\frac{d}{dt}(u\,{\rm exp}(tP))_{t=0} =
    \frac{d}{dt}(\pi^i(u\,{\rm exp}(tP)))_{t=0}\\
    &=\frac{d}{dt}(\sum_j u_j(\delta_{ji} + tP_{ji} + \textrm{higher order terms})_{t=0}\\
    &=P(u_i)^{v,TM}_{u_i},
\end{align*}
where $P(u_i) = \sum_j P_{ji}u_j$. This proves the second equality.
\end{proof}

From above lemma we have two important corollaries.
\begin{cor}
We have:
\begin{enumerate}
    \item Each $\pi^i_{\ast}$ is an isomorphism of the horizontal distribution $\mathcal{H}^{L(M)}_u$ of $L(M)$ or $\mathcal{H}^{O(M)}$ of $O(M)$ onto the horizontal distribution $\mathcal{H}^{TM}_{u_i}$ of $TM$.
    \item The $n$--tuple $(\pi^1_{\ast},\pi^2_{\ast},\ldots,\pi^n_{\ast})$ defines an isomorphism of the vertical distribution $\mathcal{V}^{L(M)}_u$ of $L(M)$ onto the product $\mathcal{V}^{TM}_{u_1}\times\mathcal{V}^{TM}_{u_2}\times\ldots\times\mathcal{V}^{TM}_{u_n}$.
\end{enumerate}
\end{cor}

Consider on the tangent bundle $TM$ the following Sasaki-Mok metric:
\begin{align*}
    g_{TM}(X^{h,TM},Y^{h,TM}) &= g_M(X,Y), \\
    g_{TM}(X^{H,TM}, Y^{v,TM}) &= 0, \\
    g_{TM}(X^{v,TM}, Y^{v,TM} &= g_M(X,Y).
\end{align*}

\begin{cor}
    For $i=1, 2, \ldots, n$, each map $\pi^i:L(M)\to M$ is a Riemannian submersion, i.e., $\pi^i_{\ast}$ is an isometry on the horizontal distribution, where
    \begin{align*}
        V^{\pi^i} & = \{ P^{\ast} \mid P\in{\rm End}(TM), P(u_i) = 0 \},\\
        \mathcal{H}^{\pi^i} &= \mathcal{H}^{L(M)}\oplus \{ P^{\ast}\mid \textrm{$P(u_j) = 0$ for $j\neq i$}\}.
    \end{align*}
\end{cor}
\begin{proof}
    Follows immediately by the definition of Riemannian metrics on $L(M)$ and $TM$ and by Lemma \ref{lem:pii}.
\end{proof}

Restricting $\pi^i$ to $O(M)$ or $O(D)$, we may rewrite above fact with natural modifications. In particular, for a bundle $O(D)$ we have $\pi^A:O(D)\to D$ and $\pi^{\alpha}:O(D)\to D^{\bot}$ and, focusing on the more important case for us, the horizontal distribution of $\pi^A$ equals
\begin{equation*}
    \mathcal{H}^{\pi^A} = \mathcal{H}^{O(D)} \oplus \{P^{\ast}\mid \textrm{$P\in\alg{g}(TM)$ such that $\skal{P}{Q}=0$ for $Q$: $Q(u_A)=0$}\}.
\end{equation*}

\section{Lifts of maps}

We begin with basic definitions for maps (submersions) between Riemannian manifolds and with results concerning maps between tangent bundles \cite{kn2}. These will be crucial in our considerations.

Let us begin with notation, which will be used throughout this section. Let $\varphi:M\to N$ be a smooth map (submersion) between Riemannian manifolds $(M,g_M)$ and $(N,g_N)$. The Levi-Civita connections of $g_M, g_N$ will be denoted, respectively, by $\nabla^M$ and $\nabla^N$.

\subsection{The second fundamental form}

There is a unique connection $\nabla^{\varphi}$ in the pull-back bundle $\varphi^{-1}TN\to M$ defined by condition (see \cite{bw})
\begin{equation*}
\nabla^{\varphi}_X(Y\circ\varphi)=\nabla^N_{\varphi_{\ast}X}Y,\quad X\in T_xM, Y\in\Gamma(TN).
\end{equation*}
The {\it second fundamental form} of $\varphi$ is a bilinear form $\Pi_{\varphi}=\nabla\varphi_{\ast}$,
\begin{equation*}
\Pi_{\varphi}(X,Y)=\nabla^{\varphi}_X\varphi_{\ast}Y-\varphi_{\ast}(\nabla^M_XY)\quad X,Y\in\Gamma(TM).
\end{equation*}
$\Pi_{\varphi}$ is symmetric and tensorial in both variables. If $\varphi$ is a diffeomorphism, then $\Pi_{\varphi}$ can be rewritten in the form $\Pi_{\varphi}(X,Y)=\nabla^N_{\varphi_{\ast}X}\varphi_{\ast}Y-\varphi_{\ast}(\nabla^M_XY)$,
hence it measures how far is $\varphi$ from being affine. For any map $\varphi$, we say that $\varphi$ is {\it totally geodesic} if $\Pi_{\varphi}$ vanishes. 

\subsection{Lift of a map to the tangent bundle}
There is a connection between the geometry of the maps between the frame bundles and the geometry of maps between tangent bundles. The correlation is given by the maps $\pi^i:L(M)\to TM$, $\pi^i(u)=u_i$, which have already been mentioned. Let us elaborate on this. 

Let $\varphi:M\to N$ be a submersion. Let $\mathcal{V}^{\varphi}={\rm ker}\varphi_{\ast}$ be the vertical distribution of $\varphi$. Denote by $\mathcal{H}^{\varphi}$ its orthogonal complement. Thus, we have a decomposition $TM=\mathcal{H}^{\varphi}\oplus\mathcal{V}^{\varphi}$. Hence, each vector $X\in TM$ can be decomposed as
\begin{equation*}
    X = X^{\top} + X^{\bot},\quad X^{\top}\in\mathcal{H}^{\varphi}, X^{\bot}\in\mathcal{V}^{\varphi}.
\end{equation*}
Moreover, at each $x\in M$, the differential $\varphi_{\ast x}:\mathcal{H}^{\varphi}_x\to T_{\varphi(x)}N$ is a linear isomorphism.

Consider the second differential $\varphi_{\ast\ast}:TTM\to TTN$. This map has been extensively studied by the first author and W. Kozlowski \cite{kn1,kn2}. Let us recall important observations from \cite{kn1}, which will be needed later.

\begin{lem}\cite{kn1}\label{lem:varphiTM}
Let $Z\in TM$. Then
\begin{align*}
\varphi_{\ast\ast}X^v_Z &= (\varphi_{\ast}X)^v_{\varphi_{\ast}Z},\\
\varphi_{\ast\ast}X^h_Z &= (\varphi_{\ast}X)^h_{\varphi_{\ast}Z} + \Pi_{\varphi}(X,Z)^v_{\varphi_{\ast}Z}.
\end{align*}
\end{lem}

We will provide the proof of this result, since it was not included in \cite{kn1}.
\begin{proof}
The first relation is straightforward. Namely,
\begin{equation*}
\varphi_{\ast\ast}X^v_Z = \varphi_{\ast\ast}\frac{d}{dt}(t\mapsto Z+ tX) = 
\frac{d}{dt}(t\mapsto \varphi_{\ast}Z + t\varphi_{\ast}X) = (\varphi_{\ast}X)^{v,TM}_{\varphi_{\ast}Z}. 
\end{equation*}
For the second relation, let $\gamma:I\to M$ be a curve such that $\dot{\gamma}(0) = X$ and let $Z:I\to TM$ be its horizontal lift to $T_ZTM$, i.e., $\dot{Z}(0)=X^h_Z$. Then $K(\dot{Z}(0)) = (\nabla_X Z)(0)$. Moreover for a curve $U=\varphi_{\ast}\circ Z$ on $N$ we have $\dot{U}(0)=\varphi_{\ast\ast} X^{h,TM}_Z$. On the other hand,
\begin{equation*}
K(\dot{U}(0)) = (\nabla_{\varphi_{\ast}X}U)(t) = \Pi_{\varphi}(X, Z).
\end{equation*}
Since,
\begin{equation*}
\pi_{TN\ast}(\varphi{\ast\ast}X^{h,TM}_Z) = \varphi_{\ast}\pi_{TM\ast}X^{h,TM}_Z = \pi_{\ast}X,
\end{equation*}
the result follows.
\end{proof}

\begin{thm}[\cite{kn1}]\label{thm:kn}
Equip $TM$ and $TN$ with Sasaki--Mok metrics.
\begin{enumerate}
    \item If $\varphi:M\to N$ is a submersion, then $\Phi = \varphi_{\ast}:TM\to TN$ is a submersion with
    \begin{align*}
        \mathcal{V}^{\Phi}_Z &=\{X^{v,TM}_Z \mid X\in\mathcal{V}^{\varphi}\}\oplus\{X^{h,TM}_Z + ((\nabla_ZX)^{\top})^{v,TM}_Z\mid X\in\mathcal{V}^{\varphi}\},\\
        \mathcal{H}^{\Phi}_Z &= \{X^{h,TM}_Z \mid X\in\mathcal{H}^{\varphi}\}\oplus
        \{X^{v,TM}_Z + ((\nabla_Z X)^{\bot})^{h,TM}\mid X\in\mathcal{H}^{\varphi}\}.
    \end{align*}
    \item The lift $\Phi = \varphi_{\ast}:TM\to TN$ is horizontally conformal if and only if $\varphi:M\to N$ is horizontally conformal with constant dilatation and totally geodesic. Moreover, the dilatation $\lambda_{\Phi}$ of $\Phi$ is constant and equal to the dilatation $\lambda_{\varphi}$, the vertical distribution $\mathcal{V}^{\varphi}$ is totally geodesic and the horizontal distribution $\mathcal{H}^{\varphi}$ is integrable. 
\end{enumerate}
\end{thm}

\subsection{Lift of a submersion to the frame bundle}

Now we introduce the main object of interest. We begin with the simplest case.
Let $\varphi:M\to N$ be a (local) diffeomorphism. Then its lift $:L\varphi:L(M)\to L(N)$ is defined as follows
\begin{equation*}
L\varphi(u)=(\varphi_{\ast}u_1,\ldots,\varphi_{\ast}u_n),\quad u=(u_1,\ldots,u_n)\in L(M).
\end{equation*}
We may alternatively consider $L\varphi:O(M)\to L(N)$. It turns out that it is possible to define the lift of any submersion. Let $\varphi:M\to N$ be a submersion and $\mathcal{H}^{\varphi}\subset TM$ be its horizontal distribution. We may consider the lift $L\varphi:L(\mathcal{H}^{\varphi})\to L(N)$ given in a natural way, namely,
\begin{equation*}
L\varphi(u)=(\varphi_{\ast}u_1,\ldots,\varphi_{\ast}u_k),\quad u=(u_1,\ldots,u_n)\in L(\mathcal{H}^{\varphi}),
\end{equation*}
where $k=\dim \mathcal{H}^{\varphi}$. We may alternatively restrict to the bundle $O(\mathcal{H}^{\varphi})$, i.e., consider $L\varphi:O(\mathcal{H}^{\varphi}\to L(N)$. In fact we will deal with the restriction of $L\varphi$ to $O(\mathcal{H}^{\varphi})$ only.

By Lemma \ref{lem:varphiTM} we have an important result. Firstly, let us introduce necessary objects and notation:
\begin{enumerate}
    \item For a tangent vector $X\in TM$, we will write
    \begin{equation*}
        X = X^{\top} + X^{\bot}
    \end{equation*}
    according to the decomposition $TM = \mathcal{H}^{\varphi}\oplus\mathcal{V}^{\varphi}$.
    \item  Instead of writing $X^{h,\mathcal{H}^{\varphi}}$ for a horizontal lift of a vector $X$ to $L(\mathcal{H}^{\varphi})$, we will write $X^{h,\varphi}$.
    \item For $P_0\in{\rm End}(\mathcal{H}^{\varphi})$, the push-forward $\varphi_{\ast}P_0\in{\rm End}(TN)$ is well defined by the condition
    \begin{equation*}
    (\varphi_{\ast}P_0)(\varphi_{\ast}X) = \varphi_{\ast}(P_0(X)),\quad X\in\mathcal{H}^{\varphi}.
    \end{equation*}
    \item Let $A_Y\in{\rm End}(\mathcal{H}^{\varphi})$ for $Y\in \mathcal{V}^{\varphi}$ be given by
    \begin{equation*}
    A_Y(X) = S_XY = (\nabla_XY)^{\top},\quad X\in\mathcal{H}^{\varphi}.
    \end{equation*}
    Notice that, 
    \begin{equation*}
    \varphi_{\ast}A_Y(X) = \Pi_{\varphi}(X,Y) = \Pi_{\varphi}(Y,X).
    \end{equation*}
    \item The second fundamental form $\Pi_{\varphi}$ for fixed first argument $\Pi_{\varphi}(X,\cdot)$ is a linear map from $TM$ to $TN$ and since $\varphi_{\ast}:\mathcal{H}^{\varphi}\to TN$ is an isomorphism, we may define $(\Pi_{\varphi})_X\in{\rm End}(\mathcal{H}^{\varphi})$ by the property:
    \begin{equation*}
    \varphi_{\ast}((\Pi_{\varphi})_XY)=\Pi_{\varphi}(X,Y).
    \end{equation*}
    In other words,
    \begin{equation*}
    (\Pi_{\varphi})_XY = \varphi_{\ast}^{-1}(\nabla_X\varphi_{\ast}Y) - (\nabla_XY)^{\top},\quad Y\in\mathcal{H}^{\varphi}.
    \end{equation*}
\end{enumerate}

\begin{lem}\label{lem:Lphi}
Let $u=(u_1,\ldots,u_n)\in O(\mathcal{H}^{\varphi})$, $X\in \mathcal{H}^{\varphi}$, $Y\in\mathcal{V}^{\varphi}$ and $P\in\alg{g}(TM)$. Then
\begin{align*}
(L\varphi)_{\ast}X^{h,\varphi}_u &= (\varphi_{\ast}X)^h + (\varphi_{\ast}(\Pi_{\varphi})_X)^{\ast}_{L\varphi(u)},\\
(L\varphi)_{\ast}Y^{h,\varphi}_u &= (\varphi_{\ast} A_Y)^{\ast}_{L\varphi(u)},\\
(L\varphi)_{\ast}P^{\ast}_u &= (\varphi_{\ast}P^{\top})^{\ast}_{L\varphi(u)}.
\end{align*}
\end{lem}

\begin{proof}
Fix $A=1, 2, \ldots, k$. Since $\pi^A\circ L\varphi = \varphi_{\ast}\circ \pi^A$, by Lemma \ref{lem:varphiTM}, Lemma \ref{lem:XhD} and Lemma \ref{lem:pii}, it follows that for any $X\in TM$
\begin{align*}
    \pi^A_{\ast}((L\varphi)_{\ast}X^{h,\varphi}_u) &= (\varphi_{\ast}\circ\pi^A)_{\ast}X^{h,\varphi}_u = \varphi_{\ast\ast}(\pi^A_{\ast}X^{h,\varphi}_u)=\varphi_{\ast\ast}(\pi^A_{\ast}(X^h_u + (S_X)^{\ast}_u)\\
    &= \varphi_{\ast\ast}X^{h,TM}_{u_A} + \varphi_{\ast\ast}(S_X(u_A)))^{v,TM}_{u_A}\\
    &=(\varphi_{\ast}X)^{h,TN}_{\varphi_{\ast}(u_A)} + \Pi_{\varphi}(X,u_A)^{v,TN}_{\varphi_{\ast}(u_A)} + (\varphi_{\ast}S_X(u_A))^{v,TN}_{\varphi_{\ast}(u_A)}.
\end{align*}
Notice that $S_X(u_A)\in\mathcal{V}^{\varphi}$. Thus $\varphi_{\ast}S_X(u_A)=0$. By above considerations we get the first two equalities.

For the third one, by Lemma \ref{lem:varphiTM}, we have
\begin{equation*}
    \pi^A_{\ast}((L\varphi)_{\ast}P^{\ast}_u) = \varphi_{\ast\ast}(\pi^A_{\ast}P^{\ast}_u) = \varphi_{\ast\ast} (P(u_A))^{v,TM}_{u_A} = (\varphi_{\ast}P(u_A))^{v,TN}_{\varphi_{\ast}(u_A)}.
\end{equation*}
Now, if $P(u_A)$ is vertical, then $\varphi_{\ast}P(u_A)=0$. Thus, it is sufficient to restrict to $P^{\top}$ and use Lemma \ref{lem:pii}. This proves the second relation.
\end{proof}

Define an endomorphism $\mathcal{W}$ of $TM$ by (compare \cite{n3})
\begin{equation*}
    \mathcal{W}(X) = X + \sum_i \skal{S_{e_i}}{S_X}e_i,
\end{equation*}
where $\skal{P}{Q} = \sum_i g_M(P(e_i), Q(e_i))$ is a natural scalar product of two endomorphisms. Notice that 
\begin{equation*}
    \skal{P}{Q} = g_{O(M)}(P^{\ast}_u, Q^{\ast}_u),\quad u\in O(M),
\end{equation*}
where $P,Q\in\alg{so}(TM)$ are skew--symmetric endomorphisms.
\begin{lem}\label{lem:S}
    $\mathcal{S}$ is an isomorphism such that
    \begin{equation*}
        g_M(X,\mathcal{W}(Y)) = g_{O(M)}(X^{h,\varphi}_u,Y^{h,\varphi}_u),\quad X,Y\in TM, u\in O(M).
    \end{equation*}
\end{lem}
\begin{proof}
    We have
    \begin{align*}
        g_M(\mathcal{W}(X),Y) &=g(X,Y) + \sum_i \skal{S_{e_i}}{S_X}g(e_i,Y)\\
        &= g_M(X,Y) + \skal{S_X}{S_Y} = g_{O(M)}(X^h,Y^h) + g_{O(M)}(S_X^{\ast}, S_Y^{ast})\\
        &=g_{O(M)}(X^h + S_X^{\ast}, Y^h + S_Y^{\ast})\\
        &=g_{O(M)}(X^{h,\varphi}, Y^{h,\varphi}).
    \end{align*}
    Moreover, if $\mathcal{W}(X) = 0$, then $h(X^{h,\varphi}, Y^{h,\varphi}) = 0$ for all $Y\in TM$. Thus $X=0$, which proves that $\mathcal{W}$ is an isomorphism.
\end{proof}

For an endomorphism $C\in{\rm End}(\mathcal{H}^{\varphi})$ let
\begin{equation*}
    {\rm div}^{\bot}C = \sum_A (\nabla_{e_A}C(e_A))^{\bot} \in\mathcal{V}^{\varphi}.
\end{equation*}
\begin{lem}\label{lem:divC}
    For $X\in \mathcal{V}^{\varphi}$ and $C\in{\rm End}(\mathcal{H}^{\varphi})$ we have
    \begin{equation*}
        \skal{A_X}{C} = -g_M(X,{\rm div}^{\bot}C). 
    \end{equation*}
\end{lem}
\begin{proof}
    By definition of $A_X$ we have
    \begin{align*}
        \skal{A_X}{C} &= \sum_A g_M(A_X(e_A), C(e_A)) \\
        &= \sum_A g_M(\nabla_{e_A}X, C(e_A)) \\
        &= -\sum_A g_M(X,(\nabla_{e_A}C(e_A))^{\bot}) \\
        &= -g_M(X, {\rm div}^{\bot}C).
    \end{align*}
\end{proof}

The consequence of lemma \ref{lem:Lphi} is the following important observation.

\begin{prop}\label{prop:Lvarphi}
If $\varphi:M\to N$ is a submersion, then $L\varphi:O(\mathcal{H}^{\varphi})\to L(N)$ is a submersion. Moreover,
\begin{align*}
    \mathcal{V}^{L\varphi}_u &= \{X^{h,\varphi}_u - (A_X)_u^{\ast} \mid X\in\mathcal{V}^{\varphi}\} \oplus \{P^{\ast}_u\mid P\in{\rm End}(\mathcal{V}^{\varphi})\cap \alg{g}(TM)\}, \\
        \mathcal{H}^{L\varphi}_u &= \{\mathcal{W}^{-1}(Y))^{h,\varphi}_u \mid Y\in\mathcal{H}^{\varphi}\} \oplus \{(\mathcal{W}^{-1}({\rm div}^{\bot}C))^{h,\varphi} - C^{\ast}\mid C\in{\rm End}(\mathcal{H}^{\varphi})\cap\alg{g}(TM)\}.
    \end{align*}
\end{prop}
\begin{proof}
    Firstly, notice that if $\varphi$ has maximal rank, so does $L\varphi$, by Lemma \ref{lem:Lphi}. Moreover, by the same lemma, we obtain the formula for the vertical distribution of $L\varphi$. 

    Let $X\in\mathcal{V}^{\varphi}$ and $P\in{\rm End}(\mathcal{V}^{\varphi})$. For the first part, for $Y\in\mathcal{H}^{\varphi}$, by Lemma \ref{lem:S} we have
    \begin{equation*}
        g_{O(M)}((\mathcal{W}^{-1}(Y))^{h,\varphi}, X^{h,\varphi} - (A_X)^{\ast} + P^{\ast}) = g_{O(M)}((\mathcal{W}^{-1}(Y))^{h,\varphi}, X^{h,\varphi}) = g_M(X,Y) = 0.
    \end{equation*}
    For the second part, for $C\in{\rm End}(\mathcal{H}^{\varphi})$, by Lemma \ref{lem:divC} we have
    \begin{align*}
        g{O(M)}((\mathcal{W}^{-1}({\rm div}^{\bot}C))^{h,\varphi} - C^{\ast}, X^{h,\varphi} - (A_X)^{\ast} + P^{\ast}) &=  g_{O(M)}((\mathcal{W}^{-1}({\rm div}^{\bot}C))^{h,\varphi}, X^{h,\varphi})\\
        &+ g_M(C^{\ast},A_X)\\ 
        &= g_M({\rm div}^{\bot}C), X) + \skal{C}{A_X}\\ 
        &= 0.
    \end{align*}
    Thus $\mathcal{H}^{L\varphi}$ is orthogonal to $\mathcal{V}^{L\varphi}$.
\end{proof}

\section{Main results} 

We study conformality and harmonicity of the lift $L\varphi$. We begin with conformality, since the latter conditions are partial consequences of the former.

\subsection{Conformality of a lift}

Let $\varphi:M\to N$ be a submersion between Riemannian manifolds and let $L\varphi:O(\mathcal{H}^{\varphi})\to L(N)$ be its lift to frame bundles. Equip $O(M)$ and $L(N)$ with natural Riemannian metrics considered in the previous section.

We say that a submersion $\varphi:M\to N$ is a {\it (horizontally) conformal map} if there is a smooth positive function $\lambda$ such that
\begin{equation*}
    g_N(\varphi_{\ast}X,\varphi_{\ast}Y) = \lambda g_M(X,Y),\quad X,Y\in\mathcal{H}^{\varphi}.
\end{equation*}
Function $\lambda$ is often called the {\it dilatation}. If $L\varphi$ is a horizontally conformal map its dilatation will be denoted by $\Lambda$.

Let us recall a basic but important fact about horizontally conformal submersions: If $\varphi:M\to N$ and $\psi:N\to F$ are horizontally conformal, then $\psi\circ\varphi:M\to F$ is horizontally conformal with dilatation $\lambda_{\psi\circ\varphi} = \lambda_{\psi}\lambda_{\varphi}$. Moreover, the vertical distribution of $\psi\circ\varphi$ equals
\begin{equation*}
    \mathcal{V}^{\psi\circ\varphi} = \mathcal{V}^{\varphi}\oplus \varphi_{\ast}^{-1}(\mathcal{V}^{\psi}).
\end{equation*}
In particular, $\varphi_{\ast}^{-1}(\mathcal{V}^{\psi})\subset\mathcal{H}^{\varphi}$.

We may state and prove the main result of this section.

\begin{thm}\label{thm:conformality}
    The lift $L\varphi:O(\mathcal{H}^{\varphi})\to L(N)$ is a horizontally conformal submersion if and only if $\varphi:M\to N$ is horizontally conformal with constant dilatation and totally geodesic. Moreover, in such case, $\Lambda$ is constant.
\end{thm}
\begin{proof}
We divide a proof into steps:

\noindent {\it Step 1: If $L\varphi$ and $\varphi$ are conformal, then $\Lambda = \lambda\circ\pi_{O(M)}$}.

In fact, for $C\in{\rm End}(\mathcal{H}^{\varphi})$ such that ${\rm div}^{\bot}C = 0$, we have
\begin{equation*}
    \Lambda(u)\sum_i |C(e_i)|^2 = \Lambda(u)|C^{\ast}_u|^2 = |(L\varphi)_{\ast}C^{\ast}_u|^2 =|(\varphi_{\ast}C)^{\ast}_{L\varphi(u)}|^2 =
    \sum_i |\varphi_{\ast}C(v_i)|^2,
\end{equation*}
where $(e_i)$ and $(v_i)$ are orthonormal bases in $\mathcal{H}^{\varphi}$ and $TN$, respectively. Since $\varphi$ is conformal, we may take $v_i=\frac{1}{\sqrt{\lambda(x)}}e_i$, where $\pi_{O(M)}(u)=x$. Thus
\begin{equation*}
    \Lambda(u)\sum_i |C(e_i)|^2 = \lambda(x)\sum_i |C^{\rm adj}(e_i)|^2 =
    \lambda(x)\sum_i |C(e_i)|^2.
\end{equation*}
Thus $\Lambda(u)=\lambda(\pi_{O(M)}(u))$.

\noindent {\it Step 2: A commutative diagram}.

Consider the following diagram

\begin{center}
\begin{tikzcd}
O(\mathcal{H}^{\varphi}) \arrow[r, "L\varphi"] \arrow[d, "\pi^A"]
& L(N) \arrow[d, "\pi^A"] \\
\mathcal{H}^{\varphi} \arrow[r, "\varphi_{\ast}"]
& TN
\end{tikzcd}
\end{center}
We have that $\pi^A:O(\mathcal{H}^{\varphi})\to \mathcal{H}^{\varphi}$ and $\pi^A:L(N)\to TN$ are Riemannian submersions and $\varphi_{\ast}:\mathcal{H}^{\varphi}\to TN$ is conformal isometry. Moreover,
\begin{equation*}
    (L\varphi)_{\ast}^{-1}(\mathcal{V}^{\pi^A}) = \{C^{\ast}\mid C\in{\rm End}(\mathcal{H}^{\varphi}\cap\alg{g}(TM), {\rm div}^{\bot}C = 0, C(u_A) = 0 \}.
\end{equation*}
This implies by Proposition \ref{prop:Lvarphi} that $\mathcal{H}^{\pi^A\circ L\varphi}$ contains $\{\mathcal{W}^{-1}(Y)^{h,\varphi}\mid Y\in\mathcal{H}^{\varphi}\}$ and
\begin{equation*}
    \bigcap_A (L\varphi)_{\ast}^{-1}(\mathcal{V}^{\pi^A}) = \emptyset.
\end{equation*}

\noindent {\it Step 3: Neccessary condition}. 

Assume $L\varphi$ is horizontally conformal. Since $\pi^A\circ L\varphi = \varphi_{\ast}\circ \pi^A$ and $\pi^A$ for all $A$ are Riemannian submersions, by Step 2 it follows that $\varphi_{\ast}$ is horizontally conformal. By Theorem \ref{thm:kn}, $\varphi$ is horizontally conformal with constant dilatation, totally geodesic and $\mathcal{H}^{\varphi}$ is integrable, $\mathcal{V}^{\varphi}$ is totally geodesic. By Step 1, $\Lambda$ is constant.

\noindent {\it Step 4: Sufficient condition}.

Assume $\varphi$ is horizontally conformal with constant dilatation and totally geodesic. By Theorem \ref{thm:kn}, $\Phi = \varphi_{\ast}$ is horizontally conformal with constant dilatation. Since $\pi^A\circ L\varphi = \varphi_{\ast}\circ \pi^A$ and $\pi^A$ for all $A$ are Riemannian submersions, by Step 2 it follows that $L\varphi$ is horizontally conformal with constant dilatation.
\end{proof}

\subsection{Lift as a harmonic morphism}

Let $(M,g_M)$ and $(N, g_N)$ be two Riemannian manifolds. A smooth map $\varphi:M\to N$ is called {\it harmonic} if its tension field $\tau(\varphi) = {\rm tr}B_{\varphi}$ vanishes, i.e., 
\begin{equation*}
    \tau(\varphi) = \sum_i \nabla^{\varphi}_{e_i}\varphi_{\ast}e_i - \varphi_{\ast}(\nabla^M_{e_i}e_i) = 0.
\end{equation*}
For a real valued function $f:M\to \mathbb{R}$, $f$ is harmonic if the Laplacian of $f$ vanishes, i.e.,
\begin{equation*}
    \Delta f = {\rm div}(\nabla f) = 0.
\end{equation*}
We say that $\varphi$ is a {\it harmonic morphism} if it maps locally harmonic functions to harmonic functions. This means, that for every harmonic function $f:U\to\mathbb{R}$ on some open subset of $M$ the composition $f\circ\varphi:\varphi^{-1}(U)\to\mathbb{R}$ is a harmonic function.

Well known characterization of harmonic morphisms \cite{bf} states that a map is a harmonic morphism if and only if it is (horizontally) conformal and harmonic. Moreover, if a map $\varphi$ is already horizontally conformal (with dilatation $\lambda$), then its tension field simplifies to
\begin{equation*}
    \tau(\varphi) = -\frac{n-2}{2}\varphi_{\ast}\nabla(\ln \lambda) - \varphi_{\ast}(H_{\varphi}),
\end{equation*}
where $H_{\varphi}$ is the mean curvature of the fibers of $\varphi$,
\begin{equation*}
    H_{\varphi} = \sum_{\alpha} (\nabla^M_{e_{\alpha}}e_{\alpha})^{\top}. 
\end{equation*}
Thus a homothety $\varphi$ is a harmonic morphism if and only if the fibers are minimal submanifolds.

We have the following characterization of lifts $L\varphi$ being harmonic morphisms.

\begin{thm}\label{thm:Liftharmmorph}
    A Lift $L\varphi:O(\mathcal{H}^{\varphi})\to L(N)$ is a harmonic morphism if and only if $\varphi$ is a totally geodesic harmonic morphism and the dilatation of $\varphi$ is constant. In particular, the dilatation of $L\varphi$ is constant.
\end{thm}
\begin{proof}
    Firstly, notice that fibers of $L\varphi$ are totally geodesic. In fact, since $\nabla^{O(\mathcal{H}^{\varphi})}_{P^{\ast}}Q^{\ast} = -\frac{1}{2}[P,Q]^{\ast}\in \mathcal{V}^{O(\mathcal{H}^{\varphi})}$ for $P,Q\in\alg{g}(TM)$, it follows that the orthogonal component vanishes.

    Assume $L\varphi$ is a harmonic morphism. In particular it is horizontally conformal. By Theorem \ref{thm:conformality}, $\varphi$ is horizontally conformal with constant dilatation and totally geodesic. Moreover, the dilatation $\Lambda$ of the lift is constant and the vertical distribution $\mathcal{V}^{\varphi}$ of $\varphi$ is totally geodesic. In particular, it is minimal. Hence, $\varphi$ is harmonic, thus a harmonic morphism.

    Conversely, assume $\varphi$ is a harmonic morphism, totally geodesic and the dilatation of $\varphi$ is constant. Then $L\varphi$, by Theorem \ref{thm:conformality} is horizontally conformal with constant dilatation. Since the fibers of $L\varphi$ are totally geodesic, hence minimal, it follows that the lift $L\varphi$ is a harmonic morphism.
\end{proof}

\end{document}